\documentclass[oneside,english]{amsart}
\usepackage[T1]{fontenc}
\usepackage[latin9]{inputenc}
\usepackage{amsthm}
\usepackage{amssymb}
\usepackage{esint}

\makeatletter
\numberwithin{equation}{section}
\numberwithin{figure}{section}
\theoremstyle{plain}
\newtheorem{thm}{Theorem}
  \theoremstyle{definition}
  \newtheorem{defn}[thm]{Definition}
  \theoremstyle{plain}
  \newtheorem{lem}[thm]{Lemma}
 \theoremstyle{definition}
  \newtheorem{example}[thm]{Example}
  \theoremstyle{plain}
  \newtheorem{cor}[thm]{Corollary}

\newcommand{\dist}{\mathop{\mathrm{dist}}\nolimits}

\makeatother

\usepackage{babel}

\begin{document}

\title{Metrics Based on Average Distance Between Sets}

\author{Osamu Fujita}

\address{Department of Arts and Sciences, Osaka Kyoiku University, 4-698-1
Asahigaoka, Kashiwara, Osaka 582-8582, Japan}

\email{fuji@cc.osaka-kyoiku.ac.jp}
\begin{abstract}
This paper presents a distance function between sets based on an average
of distances between their elements. The distance function is a metric
if the sets are non-empty finite subsets of a metric space. It can
be applied to produce various metric spaces on collections of sets
and will be useful for analyzing complex data sets in the fields of
computer science and information science. Its generalizations to include
the Hausdorff metric and extensions to infinite sets for treating
fuzzy sets are also discussed. 
\end{abstract}


\maketitle

\section{Introduction}

A metric defined in general topology \cite{Hart,Nagata}, based on
a natural notion of distance between points, is generally extensible
to distance between sets or more complex elements. The Hausdorff metric
is such a typical one and practically used for image data analysis
\cite{Rucklidge}, but it has some problems. In the Euclidean metric
on $\mathbb{R}$, for example, the Hausdorff distance between bounded
subsets of $\mathbb{R}$ often depends only on their suprema or infima,
no matter how the other elements of the sets are distributed within
a certain range, which means it places importance on extremes and
disregards the middle parts of the sets. This is a drawback because
it is sensitive to noises, errors and outliers in analyzing real world
data. There is a need to develop another metric that reflects the
overall characteristics of elements of the sets. 

In computer science, especially in the fields of pattern recognition,
classification, information retrieval and artificial intelligence,
it is important for data analysis to measure similarity or difference
between data objects such as documents, images and signals. If the
data objects can be represented by vectors, a conventional distance
between vectors is a proper measure in their vector space. In practice,
however, there are various data objects that should be dealt with
in the form of collections of sets, probability distributions, graph
structured data, or collections consisting of more complex data elements.
To analyze these data objects, numerous distance-like functions have
been developed \cite{Deza}, like the Mahalanobis distance and the
Kullback-Leibler divergence, even though they do not necessarily satisfy
symmetry and/or the triangle inequality. 

As a true metric, besides the Hausdorff metric, there is another type
of distance functions of sets, such as the Jaccard distance, based
on the cardinality of the symmetric difference between sets or its
variations. However, it measures only the size of the set difference,
and takes no account of qualitative differences between individual
elements. Thus, both metrics are insufficient to analyze informative
data sets in which each element has its own specific meaning. 

This paper presents a new distance function between sets based on
an average distance. It takes all elements into account. It is a metric
if the sets are non-empty finite subsets of a metric space, and includes
the Jaccard distance as a special case. By using the power means \cite{Bullen},
we obtain generalized forms that also include the Hausdorff metric.
Extensions of the metric to hierarchical collections of infinite subsets
will be useful for treating fuzzy sets and probability distributions.

\section{Preliminaries}

The metric is extended to various types of generalized metrics. To
avoid confusion in terminology, the following definition is used.
\begin{defn}
\label{def:metric}(Metric) Suppose $X$ is a set and $d$ is a function
on $X\times X$ into $\mathbb{R}$. Then $d$ is called a \emph{metric}
on $X$ if it satisfies the following conditions, for all $a,b,c\in X$,

\begin{minipage}[t]{1\columnwidth}%
\begin{description}
\item [{M1}] $d(a,b)\geq0$ (non-negativity),
\item [{M2}] $d(a,a)=0$,
\item [{M3}] $d(a,b)=0\Rightarrow a=b$,
\item [{M4}] $d(a,b)=d(b,a)$ (symmetry),
\item [{M5}] $d(a,b)+d(b,c)\geq d(a,c)$ (triangle inequality).
\end{description}
\end{minipage}
\end{defn}
The set $X$ is called a \emph{metric space} and denoted by $(X,d)$.
The function $d$ is called \emph{distance function} or simply \emph{distance}. 

The metric is generalized by relaxing the conditions as follows:
\begin{itemize}
\item If $d$ satisfies M1, M2, M4 and M5, then it is called a \emph{pseudo-metric}.
\item If $d$ satisfies M1, M2, M3 and M5, then it is called a \emph{quasi-metric}. 
\item If $d$ satisfies M1, M2, M3 and M4, then it is called a \emph{semi-metric}. 
\end{itemize}
This terminology follows \cite{Hart,Nagata}, though the term \emph{{}``semi-metric''}
is sometimes referred to as a synonym of \emph{pseudo-metric} \cite{Deza}. 

A set-to-set distance is usually defined as follows (see, e.g., \cite{Searcoid}):
Let $A$ and $B$ be two non-empty subsets of $X$. For each $x\in X$,
the distance from $x$ to $A$, denoted by $\dist(x,A)$, is defined
by the equation\begin{equation}
\dist(x,A)=\inf\{d(x,a)\mid a\in A\}.\label{eq:dist-p2s}\end{equation}
This is fundamental not only to the definitions of a boundary point
and an open set in metric spaces but also to the generalization of
a metric space to \emph{approach space} \cite{Lowen}. Similarly,
the distance from $A$ to $B$ can be straightforwardly defined by
\begin{equation}
\dist(A,B)=\inf\{d(a,b)\mid a\in A,b\in B\}.\label{eq:dist-s2s}\end{equation}
The function $\dist()$ is neither a pseudo-metric nor a semi-metric.
However, let $\mathcal{S}(X)$ be the collection of all non-empty
closed bounded subsets of $X$. Then, for $A,B\in\mathcal{S}(X)$,
the function $h(A,B)$ defined by\begin{equation}
h(A,B)=\max\{\sup\{\dist(b,A)\mid b\in B\},\sup\{\dist(a,B)\mid a\in A\}\}\label{eq:dist-h}\end{equation}
is a metric on $\mathcal{S}(X)$, and \emph{h} is called the \emph{Hausdorff
metric}. The collection $\mathcal{S}(X)$ topologized by the metric
$h$ is called a \emph{hyperspace} in general topology.

In computer science, data sets are generally discrete and finite.
A popular metric is the \emph{Jaccard distance} (or \emph{Tanimoto
distance,} \emph{Marczewski\textendash{}Steinhaus distance} \cite{Deza})
that is defined by\begin{equation}
j(A,B)=\frac{\left|A\triangle B\right|}{\left|A\cup B\right|},\label{eq:dist-j}\end{equation}
where $\left|A\right|$ is the cardinality of $A$, and $\triangle$
denotes the symmetric difference: $A\triangle B=(A\setminus B)\cup(B\setminus A)$.
In addition, $\left|A\triangle B\right|$ is also used as a metric.

In cluster analysis \cite{Everitt}, the distance \eqref{eq:dist-s2s}
is used as the\emph{ minimum distance} between data clusters for single-linkage
clustering, and likewise the\emph{ maximum distance} is defined by
replacing infimum with maximum for complete-linkage clustering. Moreover,
the \emph{group-average distance} (or \emph{average distance,} \emph{mean
distance}) defined as $g(A,B)$ in the following is also typically
used for hierarchical clustering. Although these three distance functions
are not metrics, the\emph{ }group-average distance\emph{ }plays an
important role in this paper.
\begin{lem}
\label{lem:g-tri} Suppose $(X,d)$ is a non-empty metric space. Let
$\mathcal{S}(X)$ denote the collection of all non-empty finite subsets
of $X$. For each $A$ and $B$ in $\mathcal{S}(X)$, define $g(A,B)$
on $\mathcal{S}(X)\times\mathcal{S}(X)$ to be the function \begin{equation}
g(A,B)=\frac{1}{\left|A\right|\left|B\right|}\sum_{a\in A}\sum_{b\in B}d(a,b).\label{eq:dist-g}\end{equation}
Then g satisfies the triangle inequality.\end{lem}
\begin{proof}
The triangle inequality for $d$ yields $d(a,b)+d(b,c)-d(a,c)\geq0$
for all $a,b,c\in X$. Then, for all $A,B,C\in\mathcal{S}(X)$, we
have \begin{align}
\lefteqn{g(A,B)+g(B,C)-g(A,C)}\label{eq:g-tri}\\
 & =\frac{1}{\left|A\right|\left|B\right|\left|C\right|}\sum_{a\in A}\sum_{b\in B}\sum_{c\in C}(d(a,b)+d(b,c)-d(a,c))\geq0.\nonumber \end{align}

\end{proof}
For ease of notation, let $s\left(A,B\right)$ be the sum of all pairwise
distances between $A$ and $B$ such that \begin{equation}
s(A,B)=\sum_{a\in A}\sum_{b\in B}d(a,b),\label{eq:sum-of-d}\end{equation}

\noindent so that $g(A,B)=(\left|A\right|\left|B\right|)^{-1}s(A,B)$.
Since $d$ is a metric, we have $s(A,B)\geq0$, $s(A,B)=s(B,A)$,
and $s(\{x\},\{x\})=0$ for all $x\in X$. If $A=\emptyset$ or $B=\emptyset$,
then $s(A,B)=0$ due to the empty sum. If $A$ and $B$ are countable
unions of disjoint sets, it can be decomposed as follows: \begin{equation}
s\Bigl(\bigcup_{i}^{n}A_{i},\bigcup_{j}^{m}B_{j}\Bigr)=\underset{i}{\sum^{n}}\underset{j}{\sum^{m}}s(A_{i},B_{j}),\label{eq:s-decomp}\end{equation}
where $A_{i}\cap A_{j}=\emptyset=B_{i}\cap B_{j}$ for $i\neq j$.
Furthermore, we define $t(A,B,C)$ by the following equation \begin{equation}
t(A,B,C)=\left|C\right|s(A,B)+\left|A\right|s(B,C)-\left|B\right|s(A,C).\label{eq:t-tri}\end{equation}
It follows from Lemma \ref{lem:g-tri} that $t(A,B,C)\geq0$ for $A,B,C\in\mathcal{S}(X)$,
which is a shorthand notation of the triangle inequality \eqref{eq:g-tri}.

\section{Metric based on average distance}
\begin{thm}
\label{thm:f-is-m}Suppose $(X,d)$ is a non-empty metric space. Let
$\mathcal{S}(X)$ denote the collection of all non-empty finite subsets
of $X$. For each $A$ and $B$ in $\mathcal{S}(X)$, define $f(A,B)$
on $\mathcal{S}(X)\times\mathcal{S}(X)$ to be the function\begin{equation}
f(A,B)=\frac{1}{\left|A\cup B\right|\left|A\right|}\sum_{a\in A}\sum_{b\in B\setminus A}d(a,b)+\frac{1}{\left|A\cup B\right|\left|B\right|}\sum_{a\in A\setminus B}\sum_{b\in B}d(a,b).\label{eq:dist-f}\end{equation}
Then $f$ is a metric on $\mathcal{S}(X)$.\end{thm}
\begin{proof}
The function $f$ can be rewritten, using $s$ in \eqref{eq:sum-of-d},
as \[
f(A,B)=\frac{s(A,B\setminus A)}{\left|A\cup B\right|\left|A\right|}+\frac{s(A\setminus B,B)}{\left|A\cup B\right|\left|B\right|}.\]
It is non-negative and symmetric. If $A=B$, then $s(A,B\setminus A)=s(A,\emptyset)=0$
and $s(A\setminus B,B)=s(\emptyset,B)=0$, so that $f(A,B)=0$. Conversely,
if $f(A,B)=0$, then $s(A,B\setminus A)=0=s(A\setminus B,B)$ for
$A,B\in\mathcal{S}(X)$. This holds if, and only if, $B\setminus A=\emptyset=A\setminus B$,
which implies $B\subseteq A$ and $A\subseteq B$. Then we have $f(A,B)=0\Leftrightarrow A=B$. 

The triangle inequality is straightforwardly proved to be $f(A,B)+f(B,C)-f(A,C)\geq0$
by showing that the left-hand terms are transformed into the sum of
non-negative terms of $s$ and $t$ in \eqref{eq:t-tri}. Let $A\cup B\cup C$
be decomposed into five disjoint partitions: $\alpha=A\setminus(B\cup C)$,
$\beta=B\setminus(A\cup C)$, $\gamma=C\setminus(A\cup B),$ $\zeta=A\cap C\setminus B$,
and $\theta=B\setminus\beta=B\cap(A\cup C)$. Then we have \begin{align*}
\lefteqn{\left|A\right|\left|B\right|\left|C\right|\left|A\cup B\right|\left|B\cup C\right|\left|A\cup C\right|\bigl(f(A,B)+f(B,C)-f(A,C)\bigr)}\\
 & =\left|B\right|\left|C\right|\bigl(\left|\theta\cup C\right|t(A,B\setminus A,\gamma)+\left|\alpha\right|t(A,\beta,\gamma)+\left|B\cup\zeta\right|t(A,\beta,C\setminus A)\bigr)\\
 & \quad+\left|A\right|\left|B\right|\bigl(\left|A\cup\theta\right|t(\alpha,B\setminus C,C)+\left|\gamma\right|t(\alpha,\beta,C)+\left|B\cup\zeta\right|t(A\setminus C,\beta,C)\bigr)\\
 & \quad+\left|A\right|\left|C\right|\bigl(\left|A\setminus B\right|t(\alpha,B,C\setminus B)+\left|B\right|t(\alpha,\theta,\gamma)+\left|C\setminus B\right|t(A\setminus B,B,\gamma)\bigr)\\
 & \quad+\left|A\right|\left|C\right|\left|\theta\right|\bigl(t(\alpha,B,\gamma)+t(\alpha,B,\zeta)+t(\zeta,B,\gamma)\bigr)\\
 & \quad+2\left|A\right|\left|C\right|(\left|\theta\cup C\right|\left|A\cup\theta\right|+\left|\beta\right|\left|A\cup C\right|)s(B,\zeta)\\
 & \quad+2\left|A\right|\left|B\right|\left|C\right|\left|B\cup\zeta\right|s(\beta,A\cap C)\geq0.\end{align*}
The details are given in Appendix A.
\end{proof}
The function $f$ in \eqref{eq:dist-f} can be rewritten, using $g$
in \eqref{eq:dist-g}, as \begin{equation}
f(A,B)=\frac{\left|B\setminus A\right|}{\left|A\cup B\right|}\, g(A,B\setminus A)+\frac{\left|A\setminus B\right|}{\left|A\cup B\right|}\, g(A\setminus B,B).\label{eq:f-by-g}\end{equation}
In $\bigl(\mathcal{S}(X),f\bigr)$, for all $a,b\in X,$ we have $f(\{a\},\{b\})=d(a,b)$
so that $\{\{x\}\mid x\in X\}$ is an isometric copy of $X$. If $A\cap B=\emptyset$,
then $f(A,B)=g(A,B)$. If $d$ is a pseudo-metric, then so is $f$.
\begin{example}
\label{exa:f-eq-J} If $d$ is the discrete metric, where $d(x,y)=0$
if $x=y$ and $d(x,y)=1$ otherwise, then $f(A,B)$ is equal to the
Jaccard distance \eqref{eq:dist-j}.\end{example}
\begin{cor}
Suppose $(X,d)$ is a non-empty metric space. Let $\mathcal{S}(X)$
denote the collection of all non-empty finite subsets of $X$. For
each $A$ and $B$ in $\mathcal{S}(X)$, define $e(A,B)$ on $\mathcal{S}(X)\times\mathcal{S}(X)$
to be the function \[
e(A,B)=\frac{1}{\left|A\right|\left|B\right|}\left(\sum_{a\in A}\sum_{b\in B}d(a,b)-\sum_{a\in A\cap B}\sum_{b\in A\cap B}d(a,b)\right).\]
Then $e$ is a semi-metric on $\mathcal{S}(X)$. \end{cor}
\begin{proof}
Let $e(A,B)$ be rewritten as \[
e(A,B)=\frac{1}{\left|A\right|\left|B\right|}\bigl(s(A\setminus B,B\setminus A)+s(A\cap B,B\setminus A)+s(A\setminus B,A\cap B)\bigr).\]
In a similar manner to the proof of Theorem \ref{thm:f-is-m}, it
can be proved that the conditions from M1 to M4, except for M5 (triangle
inequality), are satisfied. 
\end{proof}
It is noted that the triangle inequality $e(A,B)+e(B,C)\geq e(A,C)$
holds if $\left|\delta\right|\left|\varepsilon\right|\left|\eta\right|=0$,
where $\delta=A\cap B\setminus C$, $\varepsilon=B\cap C\setminus A$
and $\eta=A\cap B\cap C$. Otherwise, for example, if $A=\delta\cup\eta$,
$B=\delta\cup\eta\cup\varepsilon$ and $C=\eta\cup\varepsilon$ for
non-empty $\delta$, $\varepsilon$ and $\eta$, then we have \begin{align*}
\lefteqn{\left|A\right|\left|B\right|\left|C\right|\bigl(e(A,B)+e(B,C)-e(A,C)\bigr)}\\
 & =\left|\delta\right|s(\delta,\varepsilon)-\left|\varepsilon\right|s(\delta,\eta)-\left|\delta\right|s(\eta,\varepsilon)=-t(\delta,\eta,\varepsilon)\leq0,\end{align*}
so that the condition M5 is not generally satisfied.

\section{Extensions}

This section discusses future directions for generalization of the
average distance based on the power mean and extensions to metrics
on collections of infinite sets.

\subsection{Generalization based on the power mean}

The distance function \eqref{eq:dist-f} can be unified with the Hausdorff
metric for finite sets by using the\emph{ }power mean. To simplify
expressions, we use the following notation. Let $M_{p}^{(i)}(x\in A,\psi,w)$
be an extended weighted-power-mean of $\psi(x)$ such that \begin{equation}
M_{p}^{(1)}(x\in A,\psi,w)=\left(\frac{1}{\sum_{x\in A}w(x)}\sum_{x\in A}w(x)\bigl(\psi(x)\bigr)^{p}\right)^{1/p},\label{eq:power-m-1}\end{equation}
and its variation using the exponential transform of $\psi$, \begin{equation}
M_{p}^{(0)}(x\in A,\psi,w)=\frac{1}{p}\ln\left(\frac{1}{\sum_{x\in A}w(x)}\sum_{x\in A}w(x)\exp\bigl(p\psi(x)\bigr)\right),\label{eq:power-m-0}\end{equation}
where $i\in\{0,1\}$ indicates one of the two types \eqref{eq:power-m-1}
and \eqref{eq:power-m-0}, $p$ is an extended real number, $\psi$
is a non-negative function of $x\in A$, and $w$ is a weight such
that $w(x)\in(0,1]$ for each $x$ and $\sum_{x\in A}w(x)>0$. In
addition, let $M_{p}^{(i)}(x\in A,\psi)$ denote the abbreviation
of the equal weight case $M_{p}^{(i)}\bigl(x\in A,\psi,1_{A}(x)\bigr)$,
where $1_{A}(x)$ is the indicator function defined by $1_{A}(x)=1$
for $x\in A$ and $1_{A}(x)=0$ for $x\notin A$. If there exists
$x\in A$ such that $\psi(x)=0$ for $p<0$ in \eqref{eq:power-m-1},
then we define $M_{p}^{(1)}=0$, which is consistent with taking the
limit $\psi(x)\rightarrow0^{+}$, though such a case is undefined
in the conventional power mean to avoid division by zero. 

The power mean includes various types of means \cite{Bullen}, which
are parameterized by $p$. By taking limits also for $p=0,\pm\infty$,
we have the following: \begin{alignat*}{1}
M_{i}^{(i)}\bigl(x\in A,\psi(x),1_{A}(x)\bigr) & =M_{i}^{(i)}\bigl(x\in A,\psi(x)\bigr)=\frac{1}{\left|A\right|}\sum_{x\in A}\psi(x),\\
M_{\infty}^{(i)}\bigl(x\in A,\psi(x),w(x)\bigr) & =M_{\infty}^{(i)}\bigl(x\in A,\psi(x)\bigr)=\max\{\psi(x)\mid x\in A\},\\
M_{-\infty}^{(i)}\bigl(x\in A,\psi(x),w(x)\bigr) & =M_{-\infty}^{(i)}\bigl(x\in A,\psi(x)\bigr)=\min\{\psi(x)\mid x\in A\}.\end{alignat*}
Both $M_{1}^{(1)}$ and $M_{0}^{(0)}$ give the same arithmetic mean,
whereas $M_{0}^{(1)}$ gives the geometric mean, and neither $M_{\infty}^{(i)}$
nor $M_{-\infty}^{(i)}$ depends on $w(x)$.

There are some forms of function composition of $M_{p}^{(i)}$ that
include the distance function \eqref{eq:dist-f} as a special case,
for example, as follows: \begin{alignat}{1}
u_{p,q}^{(i,j)}(A,B) & =M_{p}^{(i)}\left(x\in A\cup B,\left(1_{B\setminus A}(x)\, M_{q}^{(j)}\bigl(y\in A,d(x,y)\bigr)\right.\right.\label{eq:u-ij-pq}\\
 & \hspace{100pt}\left.\left.+1_{A\setminus B}(x)\, M_{q}^{(j)}\bigl(y\in B,d(x,y)\bigr)\right)\right),\nonumber \end{alignat}
\begin{alignat}{1}
\lefteqn{\lefteqn{}v_{r,p,q}^{(k,i,j)}(A,B)}\label{eq:v-kij-rpq}\\
 & =M_{r}^{(k)}\left(S\in\{A,B\},M_{p}^{(i)}\left(x\in A\cup B,1_{S^{\complement}}(x)\, M_{q}^{(j)}\bigl(y\in S,d(x,y)\bigr)\right)\right),\nonumber \end{alignat}
where $i,j,k\in\{0,1\}$ and $S^{\complement}$ is the complement
of $S.$ 

In addition, let $w$ be extended to $w\in[0,1]$ to include zero,
though a weight for at least one summand must still be positive. Furthermore,
it is assumed that $0\cdot\infty=0$, in order to ensure $0\cdot0^{p}=0$
for $p<0$ in $M_{p}^{(1)}$, so that the zero-weight can be used
for excluding terms from averaging even if $\psi(x)=0$ (i.e., distance
zero) in the terms. Then, the function \eqref{eq:u-ij-pq} can be
simply expressed as \begin{equation}
u_{p,q}^{(i,j)}(A,B)=M_{p}^{(i)}\left(x\in A\cup B,M_{q}^{(j)}\bigl(y\in A\cup B,d(x,y),w(x,y)\bigr)\right),\label{eq:u-ij-pq-w}\end{equation}
by using the weight function defined by \begin{alignat*}{1}
w(x,y) & =[x\in A\setminus B][y\in B]+[x\in A\cap B][x=y]+[x\in B\setminus A][y\in A]\\
 & =1_{A\setminus B}(x)\,1_{B}(y)+1_{A\cap B}(x)[x=y]+1_{B\setminus A}(x)\,1_{A}(y),\end{alignat*}
where $[\cdot]$ denotes the Iverson bracket, that is a quantity defined
to be $1$ whenever the statement within the brackets is true, and
$0$ otherwise. 

The distance function $f$ in \eqref{eq:dist-f} and the Hausdorff
metric \eqref{eq:dist-h} are expressed by \begin{alignat*}{1}
f(A,B) & =u_{i,j}^{(i,j)}(A,B)=2\, v_{k,i,j}^{(k,i,j)}(A,B),\\
h(A,B) & =u_{\infty,-\infty}^{(i,j)}(A,B)=v_{\infty,\infty,-\infty}^{(k,i,j)}(A,B),\end{alignat*}
respectively, where $i,j,k\in\{0,1\}$. 

Although it is unclear, at present, what conditions on the parameters
$i$, $j$, $k$, $p$, $q$, and $r$ are necessary for \eqref{eq:u-ij-pq}
and \eqref{eq:v-kij-rpq} to be metrics, these generalized forms are
capable of generating various distance functions in fact as follows:
\begin{example}
\label{exa:v000rpq}The exponential types $u_{p,q}^{(0,0)}(A,B)$
and $v_{r,p,q}^{(0,0,0)}(A,B)$ are written as \begin{alignat*}{1}
u_{p,q}^{(0,0)}(A,B) & =\frac{1}{p}\ln\left(\frac{1}{\left|A\cup B\right|}\sum_{x\in A\cup B}\left(\frac{1_{B\setminus A}(x)}{\left|A\right|}\sum_{y\in A}e^{qd(x,y)}\right.\right.\\
 & \hspace{80pt}\left.\left.+1_{A\cap B}(x)+\frac{1_{A\setminus B}(x)}{\left|B\right|}\sum_{y\in B}e^{qd(x,y)}\right)^{p/q}\:\right),\end{alignat*}
\begin{alignat*}{1}
v_{r,p,q}^{(0,0,0)}(A,B) & =\frac{1}{r}\ln\left(\frac{1}{2}\Biggl(\frac{1}{\left|A\cup B\right|}\sum_{x\in B\setminus A}\biggl(\frac{1}{\left|A\right|}\sum_{y\in A}e^{qd(x,y)}\biggr)^{p/q}+\frac{\left|A\right|}{\left|A\cup B\right|}\Biggr)^{r/p}\right.\\
 & \hspace{20pt}\left.+\frac{1}{2}\Biggl(\frac{1}{\left|A\cup B\right|}\sum_{x\in A\setminus B}\biggl(\frac{1}{\left|B\right|}\sum_{y\in B}e^{qd(x,y)}\biggr)^{p/q}+\frac{\left|B\right|}{\left|A\cup B\right|}\Biggr)^{r/p}\:\right).\end{alignat*}
If $d$ is the discrete metric multiplied by a positive constant $\lambda$,
then we have \begin{align}
u_{p,q}^{(0,0)}(A,B) & =\frac{1}{p}\ln\left(\frac{e^{p\lambda}\left|A\triangle B\right|+\left|A\cap B\right|}{\left|A\cup B\right|}\right),\label{eq:u-00-pq-d}\\
v_{0,p,q}^{(0,0,0)}(A,B) & =\frac{1}{2p}\ln\left(\left(\frac{e^{p\lambda}\left|B\setminus A\right|+\left|A\right|}{\left|A\cup B\right|}\right)\left(\frac{e^{p\lambda}\left|A\setminus B\right|+\left|B\right|}{\left|A\cup B\right|}\right)\right).\label{eq:v-000-0pq-d}\end{align}
The functions \eqref{eq:u-00-pq-d} and \eqref{eq:v-000-0pq-d} are
metrics for $p>0$ and $p<0$, respectively. The proofs of each triangle
inequality are outlined in Appendix B. If $p=0$, then it is the same
situation as Example \ref{exa:f-eq-J}. By taking the limit for $p=0$,
we can see that both functions are equal to the Jaccard distance \eqref{eq:dist-j}
except for the coefficient $\lambda$.
\end{example}

\subsection{Hierarchical metric spaces\label{sub:Hier-m-sp}}

Suppose that $(X,d)$ is a metric space and $\mathcal{S}(X)$ is the
collection of all non-empty finite subsets of $X$. Let $k$ be a
non-negative integer, let $\mathcal{S}_{k+1}(X)$ denote the collection
of all non-empty finite subsets of $\mathcal{S}_{k}(X)$, and let
$f_{k}$ be a metric on $\mathcal{S}_{k}(X)$, where $\mathcal{S}_{1}(X)$,
$\mathcal{S}_{0}(X)$, $f_{1}$ and $f_{0}$ correspond to, respectively,
$\mathcal{S}(X)$, $X$, $f$, and $d$ in Theorem \ref{thm:f-is-m}.
For $k>1$, in much the same way, for each $\mathcal{A}$ and $\mathcal{B}$
in $\mathcal{S}_{k}(X)$, the function $f_{k}(\mathcal{A},\mathcal{B})$
can be defined by \begin{equation}
f_{k}(\mathcal{A},\mathcal{B})=\frac{\sum_{A\in\mathcal{A}}\sum_{B\in\mathcal{B\setminus A}}f_{k-1}(A,B)}{\left|\mathcal{A}\cup\mathcal{B}\right|\mathcal{\left|A\right|}}+\frac{\sum_{A\in\mathcal{A\setminus B}}\sum_{B\in\mathcal{B}}f_{k-1}(A,B)}{\left|\mathcal{A}\cup\mathcal{B}\right|\mathcal{\left|B\right|}},\label{eq:f-hierarchy}\end{equation}
which generates a metric space $\bigl(\mathcal{S}_{k}(X),f_{k}\bigr)$
based on $\bigl(\mathcal{S}_{k-1}(X),f_{k-1}\bigr)$. This metric
will be useful for constructing hierarchical hyperspaces.

\subsection{Duality}

There is a kind of duality between sets and elements with respect
to their distance functions. For example, we can define the functions
$D$ and $d$ symmetrically as follows: \begin{align}
D(A,B) & =\left|\{a\mid a\in A\}\triangle\{b\mid b\in B\}\right|=\left|A\triangle B\right|,\label{eq:D-symdif}\\
d(a,b) & =\left|\{A\mid a\in A\}\triangle\{B\mid b\in B\}\right|=\left|\mathcal{C}(a)\triangle\mathcal{C}(b)\right|,\label{eq:d-symdif}\end{align}
where $\mathcal{C}(a)=\{A\mid a\in A\}$. The set-to-set distance
$D$ in \eqref{eq:D-symdif} is a metric due to the axiom of extensionality,
and the element-to-element distance $d$ in \eqref{eq:d-symdif} is
a pseudo-metric. 

According to Theorem \ref{thm:f-is-m}, $D$ can be defined by $f$
in \eqref{eq:dist-f}, instead of \eqref{eq:D-symdif}, so that we
have \begin{alignat*}{1}
D(A,B) & =f(A,B),\\
d(a,b) & =\left|\mathcal{C}(a)\triangle\mathcal{C}(b)\right|.\end{alignat*}
In this case, $D$ is a pseudo-metric, depending on $d$, and there
exist a condition that satisfy $d(a,b)=f\left(\bigcap\mathcal{C}(a),\bigcap\mathcal{C}(b)\right)$.
Furthermore, in the situation of Section \ref{sub:Hier-m-sp}, let
$(X,d)$ be a metric space, let $\mathcal{S}_{1}(X)$ be the collection
of all non-empty finite subsets of $X$, and let $\mathcal{C}(a)=\{A\in\mathcal{S}_{1}(X)\mid a\in A\}$
be an element of $\mathcal{S}_{2}(X)$. If $d$ is the discrete metric,
then we have \begin{alignat*}{1}
D(A,B) & =f_{1}(A,B),\\
f_{2}\bigl(\mathcal{C}(a),\mathcal{C}(b)\bigr) & =\kappa d(a,b),\end{alignat*}
where $\kappa$ is a certain positive real number such that $\kappa<1$.
If there exists $d$ such that $d(a,b)=f_{2}\bigl(\mathcal{C}(a),\mathcal{C}(b)\bigr)$,
then the isometric copy of $(X,d)$ is contained in $\bigl(\mathcal{S}_{2}(X),f_{2}\bigr)$
and $d$ can be regarded as the function of $D$. In general, there
possibly exist $D$ and $d$ that are formally expressed by \begin{align}
D(A,B) & =F\bigl(\{d(a,b)\mid a\in A,b\in B\}\bigr),\label{eq:F-of-d}\\
d(a,b) & =G\bigl(\{D(A,B)\mid a\in A,b\in B\}\bigr),\label{eq:G-of-D}\end{align}
where $F$ and $G$ may be such a generalized function given in \eqref{eq:u-ij-pq-w}.
It is interesting to consider whether $F$ and $G$ really exist and
what features they have. In numerical analysis, $D$ and $d$ that
are consistent with each other will be obtained by the iterative computation
of \eqref{eq:F-of-d} for all $A,B\in\mathcal{S}_{1}(X)$ and \eqref{eq:G-of-D}
for all $a,b\in X$, starting with an initial metric space $(X,d)$,
if each converges to a non-trivial function.

\subsection{Generalized metrics}

The group average distance $g(A,B)$ in \eqref{eq:dist-g} can be
regarded as a generalized metric that satisfies conditions M1 (non-negativity),
M4 (symmetry) and M5 (triangle inequality) in Definition \ref{def:metric}.
In conventional topology, there has been no such generalization by
dropping both conditions M2 and M3, which are usually combined together
into the single axiom $d(a,b)=0\Leftrightarrow a=b$ (identity, reflexivity,
or coincidence). Although M3 can be dropped for the pseudo-metric
so as to allow $d(a,b)=0$ for $a\neq b$, the self-distance $d(a,a)=0$
(M2) seems to be indispensable in point-set topology where the element
is a point having no size. An exception is a \emph{partial metric}
\cite{Bukatin} that is defined to satisfy M1, M3, M4 and, instead
of M5, the following partial metric triangularity,\begin{equation}
d(a,b)+d(b,c)\geq d(a,c)+d(b,b).\label{eq:part-tri}\end{equation}

In computer science or information science, the element of data sets
is not merely a simple point. It may have rich contents inside. Some
elements may have internal structures which cause non-zero self-distance,
and some elements may have different properties each other, even though
they are indiscernible from a metric point of view. The concept of
distance can be used for measuring not only the difference between
objects but also the cost of moving or the energy of transition between
states. This is the reason why the generalization toward non-zero
self-distance is worth considering. If the triangle inequality holds,
it provides an upper and lower bound for them. The group average distance
$g(A,B)$ can be such a typical one, and that it is simpler and more
natural than the metric $f(A,B)$ in \eqref{eq:dist-f}. 

Incidentally, the function $g(A,B)$ is not a partial metric because
it does not satisfy \eqref{eq:part-tri}. On the other hand, $f(A,B)$
gives an approach to an instance of partial metrics from a special
case of \eqref{eq:v-000-0pq-d} in Example \ref{exa:v000rpq}. By
taking the limit as $\lambda\rightarrow\infty$ for $p<0$ and multiplying
a positive constant, for non-empty finite sets $A$ and $B$, we have
the following metric, \[
D_{\nu}(A,B)=\log\left|A\cup B\right|-\nu\log(\left|A\right|\left|B\right|),\]
where $\nu=1/2$. Its triangle inequality is equivalent to $\left|A\cup B\right|\left|B\cup C\right|\geq\left|A\cup C\right|\left|B\right|$.
This suggests, for $\nu\in[0,1/2)$, $D_{\nu}$ is a partial metric
on a collection of non-empty finite sets.

\subsection{Extension to infinite sets}

If $\mathcal{S}(X)$ is the collection of all non-null measurable
subsets of $(X,d)$, and $d$ is Lebesgue integrable on each element
of $\mathcal{S}(X)$, then the group average distance\emph{ }$g(A,B)$,
for $A,B\in\mathcal{S}(X)$, can be defined by \begin{equation}
g(A,B)=\frac{1}{\mu(A)\mu(B)}\int_{A}\left(\int_{B}d(x,y)\mathrm{d}\mu(y)\right)\mathrm{d}\mu(x)\label{eq:g-myu}\end{equation}
where $x\in A$, $y\in B$, and $\mu$ is a measure on $X$, and then
the distance function \eqref{eq:f-by-g} can be extended to \begin{equation}
f(A,B)=\frac{\mu(B\setminus A)}{\mu(A\cup B)}\, g(A,B\setminus A)+\frac{\mu(A\setminus B)}{\mu(A\cup B)}\, g(A\setminus B,B).\label{eq:f-myu}\end{equation}
If $d$ is the discrete metric, then \eqref{eq:f-myu} is equal to
the \emph{Steinhaus distance} \cite{Deza}.
\begin{example}
Let $(\mathbb{R},d)$ be a metric space and let $d(x,y)=\left|x-y\right|$.
For two intervals $A$ and $B$, the distance function \eqref{eq:f-myu}
can be expressed as \begin{alignat*}{1}
f(A,B) & =\frac{\left|\sup(A)-\sup(B)\right|+\left|\inf(A)-\inf(B)\right|}{2}\\
 & \quad-\frac{\left|\sup(A)-\sup(B)\right|\left|\inf(A)-\inf(B)\right|}{\sup(A\cup B)-\inf(A\cup B)}[(A\subset B)\vee(A\supset B)]\end{alignat*}
 If $A\nsubseteq B$ and $A\nsupseteq B$ (i.e., $[(A\subset B)\vee(A\supset B)]=0)$,
then $f(A,B)$ is equal to the distance between the centers of $A$
and $B$. This is consistent with an intuitive notion of the distance
between balls in this $(\mathbb{R},d)$.
\end{example}
If $\mathcal{S}(X)$ is the collection of all non-empty, countably
infinite subsets (measure-zero sets) of $X$, then $g(A,B)$ and $f(A,B)$
should be defined by taking limits in \eqref{eq:dist-g} and \eqref{eq:f-by-g},
provided that both have definite values. In order to determine the
average distance, we have to define a proper condition, which should
be said to be {}``averageable''. The average distance will strongly
depend on accumulation points in $A$ and $B$, and it will require
additional assumptions on the difference of the strength between the
accumulation points. This requirement is closely related to a {}``relative
measure'' that is needed to obtain the ratio of the cardinality of
an infinite set to the cardinality of its superset in \eqref{eq:f-by-g}.
In conventional measure theory, however, any set of cardinality $\aleph_{0}$
is a null set having measure zero so that both counting measure and
Lebesgue measure are useless for computing the ratio. It is necessary
to use another measure. A feasible solution is discussed in the following
section.

\subsection{Estimation by sampling}

In application to computational data analysis, statistical estimation
by sampling is very useful for obtaining the approximate value of
$f(A,B)$ when the size of the sets is very large. According to the
law of large numbers, if enough sample elements are selected randomly,
an average generated by those samples should approximate the average
of the total population. The procedure is as follows:
\begin{enumerate}
\item Choose a superset $P$ of $A\cup B$ as a population such that $P\supseteq A\cup B$.
\item Select a finite subset $S$ of $P$ as a sample obtained by random
sampling.
\item Let $S_{A}=S\cap A$ and $S_{B}=S\cap B$. Then, compute $f(S_{A},S_{B})$
for approximation of $f(A,B)$.
\end{enumerate}
The sampling process and its randomness are crucial for efficiently
estimating a good approximation. Some useful hints could be found
in various sampling techniques developed for Monte Carlo methods \cite{Rubinstein}.
In most cases, sampling error is expected to decrease as the sample
size increases, except for situations where the distribution of $d$
has no mean (e.g., Cauchy distribution).

The notion of sampling suggests an intuitive measure to define a relative
measure on a $\sigma$-algebra over a set $X$, which could be called
{}``sample counting measure''. Suppose $A$ and $B$ are subsets
of $X$. Let $\rho(A:B)$ be the ratio of the cardinality of $A$
to the cardinality of $B$, let $P$ be a superset of $A\cup B$,
and let $S_{n}$ be a non-empty finite subset of $P$ such that $S_{n}=\bigcup_{i=1}^{n}Y_{i}$
where $Y_{i}$ is the $i$-th non-empty sample randomly selected from
$P$. Then, the ratio $\rho(A:B)$ can be determined by taking a limit
of $n$ as it approaches to $\infty$ as follows: \[
\rho(A:B)=\underset{n\rightarrow\infty}{\lim}\frac{\left|S_{n}\cap A\right|}{\left|S_{n}\cap B\right|},\]
if there exist such a limit and a random choice function that performs
random sampling. Otherwise, instead of random sampling, systematic
sampling could be available if the elements of $X$ are supposed to
be distributed with uniform density in its measurable metric space.
For example, suppose there exists a finite partition of $P$ where
every part has an almost equal diameter. It seems better for $S_{n}$
to have exactly one element with each of the parts.

\subsection{Metrics for fuzzy sets and probability distributions}

A fuzzy set can be represented by a collection of crisp sets so that
the distance between fuzzy sets can be defined by the distance between
the collections of such crisp sets. Let $A$ be a fuzzy set: $A=\{\bigl(x,m_{A}(x)\bigr)\mid x\in X\}$,
where $m_{A}(x)$ is a membership function, and let $A_{\alpha}$
be a crisp set called an $\alpha-lebel$ set \cite{Zimmermann} such
that $A_{\alpha}=\{x\in X\mid m_{A}(x)\geq\alpha\}$. Then, $A$ can
be represented by the following set of ordered pairs: $\mathcal{C}(A)=\{(A_{\alpha},\alpha)\mid\alpha\in(0,1]\}.$
The distance between two fuzzy sets $A$ and $B$ can be defined by
$f_{2}\bigl(\mathcal{C}(A),\mathcal{C}(B)\bigr)$, where there may
be various ways to treat $\alpha$. This notion is also applicable
to the distance between probability distributions, where probability
density functions are used instead of the membership function.

\section{Concluding Remarks}

We have found that, for a metric space $(X,d)$, there exists a distance
function between non-empty finite subsets of $X$ that is a metric
based on the average distance of $d$. The distance function \eqref{eq:dist-f}
in Theorem \ref{thm:f-is-m} is the most typical one, which includes
the Jaccard distance as a special case where $d$ is the discrete
metric. Its extensions based on the power mean will be useful to develop
generalized forms that also include the Hausdorff metric and the other
various distance functions. Furthermore, the extensions to infinite
subsets of $X$ will provide metrics for measuring dissimilarity of
fuzzy sets and probability distributions.

\specialsection*{Appendix A. Triangle Inequality in Theorem \ref{thm:f-is-m}}

The triangle inequality for\[
f(A,B)=(\left|A\cup B\right|\left|A\right|)^{-1}s(A,B\setminus A)+(\left|A\cup B\right|\left|B\right|)^{-1}s(A\setminus B,B)\]
can be proved by showing the following inequality: \[
\left|A\right|\left|B\right|\left|C\right|\left|A\cup B\right|\left|B\cup C\right|\left|A\cup C\right|\bigl(f(A,B)+f(B,C)-f(A,C)\bigr)\geq0.\]
Let $A\cup B\cup C$ be decomposed into the following seven disjoint
sets: \begin{alignat*}{3}
\alpha & =A\setminus(B\cup C), & \quad\beta & =B\setminus(A\cup C), & \quad\gamma & =C\setminus(A\cup B),\\
\delta & =A\cap B\setminus C, & \varepsilon & =B\cap C\setminus A, & \zeta & =C\cap A\setminus B,\quad\eta=A\cap B\cap C,\end{alignat*}
and let $\theta=B\setminus\beta=B\cap(A\cup C)$, so that we have
\begin{alignat*}{2}
A & =\alpha\cup\delta\cup\zeta\cup\eta, & \quad\left|A\right| & =\left|\alpha\right|+\left|\delta\right|+\left|\zeta\right|+\left|\eta\right|,\\
B & =\beta\cup\delta\cup\varepsilon\cup\eta, & \left|B\right| & =\left|\beta\right|+\left|\delta\right|+\left|\varepsilon\right|+\left|\eta\right|,\\
C & =\gamma\cup\varepsilon\cup\zeta\cup\eta, & \left|C\right| & =\left|\gamma\right|+\left|\varepsilon\right|+\left|\zeta\right|+\left|\eta\right|,\\
\theta & =\delta\cup\varepsilon\cup\eta, & \left|\theta\right| & =\left|\delta\right|+\left|\varepsilon\right|+\left|\eta\right|.\end{alignat*}
Taking account of \eqref{eq:s-decomp} and \eqref{eq:t-tri}, we have
\begin{alignat*}{1}
\lefteqn{\left|A\right|\left|B\right|\left|C\right|\left|A\cup B\right|\left|B\cup C\right|\left|A\cup C\right|\bigl(f(A,B)+f(B,C)-f(A,C)\bigr)}\\
 & =\left|B\right|\left|C\right|\left|B\cup C\right|\left|A\cup C\right|s(A,B\setminus A)+\left|A\right|\left|C\right|\left|B\cup C\right|\left|A\cup C\right|s(A\setminus B,B)\\
 & \quad+\left|A\right|\left|C\right|\left|A\cup B\right|\left|A\cup C\right|s(B,C\setminus B)+\left|A\right|\left|B\right|\left|A\cup B\right|\left|A\cup C\right|s(B\setminus C,C)\\
 & \quad-\left|B\right|\left|C\right|\left|A\cup B\right|\left|B\cup C\right|s(A,C\setminus A)-\left|A\right|\left|B\right|\left|A\cup B\right|\left|B\cup C\right|s(A\setminus C,C)\\[2mm]
 & =\left|B\right|\left|C\right|\bigl(\left|\gamma\right|\left|\delta\cup C\right|s(A,B\setminus A)+(\left|\gamma\right|\left|\alpha\right|+\left|B\cup\zeta\right|\left|C\setminus A\right|)s(A,\beta)\\
 & \hspace{7mm}+\left|B\cup\zeta\right|\left|A\right|\bigl(s(A\setminus C,\beta)+s(A\cap C,\beta)\bigr)+(\left|\beta\right|\left|\gamma\right|+\left|B\cup C\right|\left|A\cup\varepsilon\right|)s(A,\varepsilon)\bigr)\\
 & \quad+\left|A\right|\left|C\right|\bigl((\left|\gamma\right|\left|A\cup C\right|+\left|\zeta\right|\left|\gamma\right|+\left|\zeta\right|\left|A\cup\varepsilon\right|)s(\alpha,B)\\
 & \hspace{11mm}+\left|B\right|\left|A\cup\varepsilon\right|(s(\alpha,B\setminus C)+s(\alpha,B\cap C))+\left|B\right|\left|\gamma\right|(s(\alpha,\beta)+s(\alpha,\theta))\\
 & \hspace{43mm}+(\left|\delta\cup C\right|\left|\gamma\right|+\left|\delta\cup C\right|\left|A\cup\varepsilon\right|+\left|\beta\right|\left|A\cup C\right|)s(\zeta,B)\bigr)\\
 & \quad+\left|A\right|\left|C\right|\bigl((\left|\alpha\right|\left|A\cup C\right|+\left|\zeta\right|\left|\alpha\right|+\left|\zeta\right|\left|\delta\cup C\right|)s(B,\gamma)\\
 & \hspace{11mm}+\left|B\right|\left|\delta\cup C\right|(s(B\setminus A,\gamma)+s(A\cap B,\gamma))+\left|B\right|\left|\alpha\right|(s(\beta,\gamma)+s(\theta,\gamma))\\
 & \hspace{43mm}+(\left|A\cup\varepsilon\right|\left|\alpha\right|+\left|A\cup\varepsilon\right|\left|\delta\cup C\right|+\left|\beta\right|\left|A\cup C\right|)s(B,\zeta)\bigr)\\
 & \quad+\left|A\right|\left|B\right|\bigl(\left|\alpha\right|\left|A\cup\varepsilon\right|s(B\setminus C,C)+(\left|\alpha\right|\left|\gamma\right|+\left|\zeta\cup B\right|\left|A\setminus C\right|)s(\beta,C)\\
 & \hspace{7mm}+\left|\zeta\cup B\right|\left|C\right|(s(\beta,A\cap C)+s(\beta,C\setminus A))+(\left|\beta\right|\left|\alpha\right|+\left|A\cup B\right|\left|\delta\cup C\right|)s(\delta,C)\bigr)\\
 & \quad-\left|B\right|\left|C\right|\bigl(\left|\zeta\cup B\right|\left|\beta\right|s(A,C\setminus A)+(\left|\alpha\right|\left|\beta\right|+\left|B\setminus A\right|\left|\delta\cup C\right|)s(A,\gamma)\\
 & \hspace{8mm}+\left|A\right|\left|\delta\cup C\right|(s(\alpha\cup\zeta,\gamma)+s(A\cap B,\gamma))+(\left|\beta\right|\left|\gamma\right|+\left|A\cup\varepsilon\right|\left|B\cup C\right|)s(A,\varepsilon)\bigr)\\
 & \quad-\left|A\right|\left|B\right|\bigl(\left|\beta\right|\left|B\cup\zeta\right|s(A\setminus C,C)+(\left|\beta\right|\left|\gamma\right|+\left|A\cup\varepsilon\right|\left|B\setminus C\right|)s(\alpha,C)\\
 & \hspace{8mm}+\left|A\cup\varepsilon\right|\left|C\right|(s(\alpha,\zeta\cup\gamma)+s(\alpha,B\cap C))+(\left|\alpha\right|\left|\beta\right|+\left|A\cup B\right|\left|\delta\cup C\right|)s(\delta,C)\bigr)\\[2mm]
 & =\left|B\right|\left|C\right|\bigl(\left|\delta\cup C\right|t(A,B\setminus A,\gamma)+\left|\alpha\right|t(A,\beta,\gamma)+\left|B\cup\zeta\right|t(A,\beta,C\setminus A)\bigr)\\
 & \quad+\left|A\right|\left|B\right|\bigl(\left|A\cup\varepsilon\right|t(\alpha,B\setminus C,C)+\left|\gamma\right|t(\alpha,\beta,C)+\left|B\cup\zeta\right|t(A\setminus C,\beta,C)\big)\\
 & \quad+\left|A\right|\left|C\right|\bigl((\left|A\cup C\right|+\left|\zeta\right|)t(\alpha,B,\gamma)+\left|B\right|t(\alpha,\theta,\gamma)\bigr)\\
 & \quad+\left|A\right|\left|C\right|\bigl(\left|A\cup\varepsilon\right|t(\alpha,B,\zeta)+\left|C\cup\delta\right|t(\zeta,B,\gamma)\bigr)\\
 & \quad+2\left|A\right|\left|C\right|(\left|\delta\cup C\right|\left|A\cup\varepsilon\right|+\left|\beta\right|\left|A\cup C\right|)s(B,\zeta)\\
 & \quad+2\left|A\right|\left|B\right|\left|C\right|\left|B\cup\zeta\right|s(\beta,A\cap C)\geq0.\end{alignat*}
 where the equality holds if all terms of $s$ and $t$ are zero.

\specialsection*{Appendix B. Triangle Inequalities of Example \ref{exa:v000rpq}}

The triangle inequality for \eqref{eq:u-00-pq-d} can be proved as
follows: Let $x=e^{p\lambda}$ and let \[
\tau(x)=\left(\frac{x\left|A\triangle B\right|+\left|A\cap B\right|}{\left|A\cup B\right|}\right)\left(\frac{x\left|B\triangle C\right|+\left|B\cap C\right|}{\left|B\cup C\right|}\right)-\left(\frac{x\left|A\triangle C\right|+\left|A\cap C\right|}{\left|A\cup C\right|}\right).\]
Then the triangle inequality for $p>0$ is equivalent to $\tau(x)\geq0$
for $x>1$. The first derivative of $\tau(x)$ with respect to $x$
is \[
\tau^{\prime}(x)=2(x-1)j(A,B)j(B,C)+j(A,B)+j(B,C)-j(A,C),\]
where $j$ is the Jaccard distance \eqref{eq:dist-j}. Since $\tau(1)=0$
and $\tau^{\prime}(x)\geq0$ for $x\geq1$, we have $\tau(e^{p\lambda})\geq0$
for $p>0$.

The triangle inequality for \eqref{eq:v-000-0pq-d} can be proved
as follows: Let $y=1-e^{p\lambda}$ and let \begin{alignat*}{1}
\tau(y) & =\left(1-\frac{\left|A\setminus C\right|}{\left|A\cup C\right|}y\right)\left(1-\frac{\left|C\setminus A\right|}{\left|A\cup C\right|}y\right)\\
 & \quad-\left(1-\frac{\left|A\setminus B\right|}{\left|A\cup B\right|}y\right)\left(1-\frac{\left|B\setminus A\right|}{\left|A\cup B\right|}y\right)\left(1-\frac{\left|B\setminus C\right|}{\left|B\cup C\right|}y\right)\left(1-\frac{\left|C\setminus B\right|}{\left|B\cup C\right|}y\right)\\
 & =y\phi(y),\end{alignat*}
where $\phi(y)$ is the cubic function of $y$ with a negative leading
coefficient. Then the triangle inequality for $p<0$ is equivalent
to $\tau(y)\geq0$ for $y\in[0,1)$. The function $\phi(y)$ satisfies
the following inequalities: $\phi(0)\geq0$, $\phi(1)=\tau(1)\geq0$,
$\phi^{\prime}(1)\leq0$, and $\phi^{\prime\prime}(1)\geq0$. These
inequalities can be proved by decomposition of $A$, $B$, and $C$
into $\alpha$, $\beta$, $\gamma$, $\delta$, $\varepsilon$, $\zeta$
and $\eta$ defined in Appendix A. Then, we have $\phi(y)\geq0$ for
$y\in[0,1)$, therefore, $\tau(1-e^{p\lambda})\geq0$ for $p<0$.

\end{document}